\documentclass{amsart}
\usepackage{amssymb}
\usepackage{amsmath}
\usepackage{amsfonts}
\theoremstyle{plain}
\newtheorem{thm}{Theorem}[section]

\newtheorem{lmm}[thm]{Lemma}
\newtheorem{cor}[thm]{Corollary}

\newtheorem{rmk}[thm]{Remark}

\newtheorem{prb}[thm]{Problem}
\theoremstyle{remark}

\def\NN{\mathbb{N}}
\def\RR{\mathbb{R}}

\def\II{\mathbb{I}}

\def\vep{\varepsilon}

\def\pmc#1{\setbox0=\hbox{#1}
    \kern-.1em\copy0\kern-\wd0
    \kern.1em\copy0\kern-\wd0}

\def\vep{\varepsilon}

\def\si{\sigma}

\def\vp{\varphi}

\def\ov{\overline}

\def\sm{\setminus}

\begin{document}
\title
{A nonaspherical cell-like 2-dimensional 
simply connected continuum and 
some
related constructions}

\author{Katsuya Eda}
\address{Department of Mathematics,
Waseda University, Tokyo 169-8555, Japan}
\email{eda@logic.info.waseda.ac.jp}

\author{Umed H. Karimov}
\address{Institute of Mathematics,
Academy of Sciences of Tajikistan,
Ul. Ainy $299^A$, Dushanbe 734063, Tajikistan}
\email{umed-karimov@mail.ru}

\author{Du\v san Repov\v s }
\address{Institute of Mathematics, Physics and Mechanics,
Faculty of Mathematics and Physics, 
University of Ljubljana, P.O.Box 2964,
Ljubljana 1001, Slovenia}
\email{dusan.repovs@fmf.uni-lj.si}

\date     
{July 20, 2008}



\begin{abstract}
We prove the existence of a 2-dimensional nonaspherical simply connected 
cell-like Peano continuum (the space itself was constructed
in one of our  earlier papers).
We also indicate some relations between this space and 
the well-known Griffiths' space from the 1950's.
\end{abstract}

\title
[Nonaspherical cell-like 2-dimensional continuum]   
{A nonaspherical cell-like 2-dimensional simply connected 
continuum and related constructions}   
\maketitle
\section{Introduction}

It is well-known (see \cite{F, KR}) that every $n$-dimensional compactum
is weakly homotopy equivalent to an $(n+1)$-dimensional cell-like 
compactum (i.e. a compactum with the trivial shape).
Therefore there exist nonaspherical cell-like simply connected 
compacta in all dimensions $\ge 3$.

It was heretofore unknown whether such compacta also exist in dimension
2. In this paper we give the affirmative answer to this question. 
We show that the space $SC(S^1)$ which we
constructed in our earlier paper \cite{EKR2},
is in fact, a {\sl nonaspherical} cell-like 2-dimensional 
simply connected {\sl Peano} continuum (i.e. locally connected continuum).

We also modify our original construction of the space
$SC(S^1)$ and show that the modified
construction gives a space which has the homotopy type of 
the classical well-known space \cite{GJ} from the 1950's, which is a
non-simply connected 
one-point 
union of two contractible spaces.

Our main result concerns $SC(X)$ for a non-simply connected
path-connected space $X$. 
To analyze the sigular homology $H_2(SC(X))$, we use
infinitary words and a result from \cite{E:free}. 
Although infinitary
words have already been
introduced in \cite{CannonConner:hawaii1}, they may not
be a familiar notion. 
In the special case $X=S^1$, we can prove the result only by using
finitary words - we present it at the end of Section~\ref{sec:proof}.

\section{Preliminaries}\label{sec:preliminary}

We recall the construction of the space $SC(S^1)$
from \cite{EKR2}.
Consider the so-called {\sl Topologist sine curve} $T$ and embed $T$
into the square $\II^2 = \II \times \II$ as in Figure 1, i.e. $T$ is
embeded as the union of $A_1B_1A_2B_2\cdots$ and
$AB$. 
Let $S^1$ be the circle and $s_0$  any of its points which we
consider as the base point. Consider the topological sum of $\II ^2$ and
$T\times S^1$. 
The space $SC(S^1)$ is
now defined as the quotient space of this sum, obtained by
identification of the points $(t, s_0)$ with $t\in T\subset \II^2$,
and by identification of each set $\{t\}\times S^1$ with $t$,
when $t\in \{ 0\}\times \II$. 
For an arbitrary compactum $X$, one defines the space $SC(X)$
by replacing $S^1$ everywhere above by $X$. 
For the details of the definition of $SC(X)$ 
we refer the reader to \cite{EKR2}.

The subspace $\mathbb{H} = \bigcup _{m=1}^\infty 
\{ (x,y):(x-1/m)^2+y^2 = 1/m^2\}$
of the Euclidean plane $\RR ^2$ is called the {\sl Hawaiian earring}. 
Denote $\theta = (0,0)\in \mathbb{H}$ and 
let $C(\mathbb{H})$ be the cone over $\mathbb{H}$. 
We consider $\mathbb{H}$ as the subspace of
$C(\mathbb{H})$. 
A space $\mathcal{G}$ is then defined as the one-point
union of two copies of $C(\mathbb{H})$, obtained by identifying
two copies of $\theta$ at the point $\theta$. 
This space is a well-known example of a
non-contractible space which is a one-point union of contractible
spaces -- Griffiths was the first to investigate this
kind of spaces \cite[p.190]{GJ}, where he also acknowleges ideas by James. 
The fact that $\mathcal{G}$ is aspherical was proved in
\cite{EK:preprint}. For further information of this space and its
generalizations we refer the reader to \cite{E:union, E:cones,
EK:aloha}.

Throughout the paper, we shall denote the singular homology
with integer coefficients by $H_*( )$.

\section{On nonasphericity of $SC(S^1)$ and $SC(X)$}\label{sec:proof}

Obviously, $SC(S^1)$ is a cell-like Peano continuum. It was shown
in \cite{EKR2} that this space is simply connected. Therefore it
suffices to show that $SC(S^1)$ is nonaspherical.
In order to prove this it certainly 
suffices to verify
that there exists a nontrivial 2-dimensional
singular cycle in $SC(S^1)$.
We shall prove this as a corollary of the following general result --
Theorem 3.1 below -- in the sense of \cite{EKR2}.
Our notation for $SC(X)$ is the same as in \cite{EKR2}. 

\setlength{\unitlength}{0.7mm}
\begin{picture}(110,110)
\put(0,51){\line(1,0){100}}
\put(26,0){\line(0,1){101}}
\put(26,0){\line(1,4){25}}
\put(51,0){\line(0,1){101}}
\put(51,0){\line(1,2){50}}
\put(1,1){\line(1,0){100}}
\put(1,101){\line(1,0){100}}
\put(1,1){\line(0,1){100}}
\put(101,1){\line(0,1){100}}

\put(39,47){\shortstack{$C_{4}$}}
\put(27,47){\shortstack{$C_{5}$}}
\put(52,47){\shortstack{$C_{3}$}}
\put(76,47){\shortstack{$C_{2}$}}
\put(103,47){\shortstack{$C_{1}$}}
\put(-3,-3){\shortstack{$A$}}
\put(-3,47){\shortstack{$C$}}
\put(-2,103){\shortstack{$B$}}
\put(102,103){\shortstack{$B_1$}}
\put(50,103){\shortstack{$B_2$}}
\put(25,103){\shortstack{$B_3$}}
\put(102,-3){\shortstack{$A_1$}} 
\put(50,-3){\shortstack{$A_2$}} 
\put(25,-3){\shortstack{$A_3$}} 
\put(100,0){\shortstack{$\bullet$}} 
\put(50,0){\shortstack{$\bullet$}} 
\put(25,0){\shortstack{$\bullet$}} 
\put(25,100){\shortstack{$\bullet$}} 
\put(50,100){\shortstack{$\bullet$}} 
\put(100,100){\shortstack{$\bullet$}}
\put(100,50){\shortstack{$\bullet$}} 
\put(75,50){\shortstack{$\bullet$}} 
\put(50,50){\shortstack{$\bullet$}} 
\put(37.5,50){\shortstack{$\bullet$}} 
\put(25,50){\shortstack{$\bullet$}} 
\put(0,50){\shortstack{$\bullet$}} 
\put(0,100){\shortstack{$\bullet$}}
\put(0,0){\shortstack{$\bullet$}} 
\end{picture}

\medskip

\qquad
Figure 1

\medskip

Consider Figure 1: the 
piecewise linear line $A_1B_1A_2B_2\cdots $ with the
segment $A B$ in this figure is the PL
Topologist sine curve which was used to build $SC(X)$,
i.e. along which we attached the ``infinite tube".

\begin{thm}\label{thm:main1}
Let $X$ be any path-connected space. Then the following assertions hold:
\begin{itemize}
\item[(1)] If $X$ is not simply connected, then $H_2(SC(X))$ is not trivial; and 
\item[(2)] If $\pi _1(X)$ and $\pi _2(X)$ are trivial, then $H_2(SC(X))$ is also trivial. 
\end{itemize}
\end{thm}

\begin{cor}\label{cor:circle}
The space
$SC(S^1)$ 
is a nonaspherical cell-like 2-dimensional simply connected Peano
continuum.
\end{cor}

For the proof of Theorem~\ref{thm:main1}, we recall a notion of
the free $\sigma$-product of groups and a lemma from \cite{E:free}. 
Let $(X_i,x_i)$ be any family of pointed spaces such that 
$X_i \cap X_j = \emptyset $, for $i \neq j.$ 
The underlying set of a pointed space 
$(\widetilde{\bigvee}_{i\in I}(X_i,x_i),x^*)$ 
is the union of all $X_i$ s obtained by identifying all $x_i $ to a point
$x^*$ and the topology is defined by specifying the neighborhood bases as 
follows:

\begin{itemize}
\item[(1)] If $x \in X_{i} \setminus \{x_{i}\}$, then the neighborhood base
of $x$ in $\widetilde{\bigvee}_{i\in I}(X_i,x_i)$ is the one of $X_i$; 
\item[(2)] The point $x^*$ has a neighborhood base, each element of which is
of the form:  $\widetilde{\bigvee}_{i\in I \setminus F}(X_i,x_i) {\vee} 
{\bigvee}_{j \in F}U_{j},$ where $F$ is a finite subset of $I$ and each $U_j$
is an open neighborhood of $x_j$ in $X_j$ for $j \in F$. 
\end{itemize}

\begin{lmm}\label{lmm:first}\cite[Theorem A.1]{E:free} 
Let $X_i$ be locally simply-connected and first countable at $x_n$ for
 each $i$. Then 
 $$\pi _1(\widetilde{\bigvee}_{i\in I}(X_i,x_i),x^*)
\simeq \pmc{$\times$}\, \, \; _{i\in I}^{\si}\pi _1(X_i,x_i).$$  
In particular $I = \mathbb{N}$, 
 $$\pi _1(\widetilde{\bigvee}_{n\in\mathbb{N}}(X_n,x_n),x^*)
\simeq \pmc{$\times$}\, \, \; _{n\in\mathbb{N}}\pi _1(X_n,x_n).$$  
\end{lmm}

We also need basic descriptions of paths and loops. 
A loop $f:\II \to X$ is a continuous map with
$f(0)=f(1)$. For a loop $f$, $f^-$ denotes the loop defined by: 
$f^-(t) = f(1-t)$. 
For loops $f,g$ with the same base point, the concatenation $fg$ is
a loop defined by: $fg(t) = f(2t)$ for $0\le t\le 1/2$ and $fg(t)
=g(2t-1)$ for $1/2\le t\le 1$. 
We denote the homotopy class relative to end points of a loop $f$ by
$[f]$ and the homology class of $f$ by $[f]_s$.

\medskip
{\it Proof of\/} Theorem~\ref{thm:main1}.
Let $p$ be the natural projection of $SC(X)$ onto
$\II^2$ which we consider as a subspace of the plane $\RR^2$.

Let $Y_0 = p^{-1}(\II\times [0,2/3))$ and $Y_1 = 
p^{-1}(\II\times (1/3,1])$. 
Then $SC(X)= Y_0\cup Y_1$ and $Y_0\cap Y_1$ is open in $SC(X)$.

Consider the following Mayer-Vietoris homology exact sequence:
\[
H_2(SC(X)) \overset{\partial}\longrightarrow 
H_1(Y_0\cap Y_1)\overset{h}\longrightarrow 
H_1(Y_0)\oplus H_1(Y_1).  
\]
We let $i_0: Y_0\cap Y_1 \to Y_0$ and $i_1:Y_0\cap Y_1\to
 Y_1$ be the inclusion maps. Then $h = i_{0*} - i_{1*}$. 
 
We now present the proof of property (1) above.
We first observe that non-injectivity of $h$ implies
that $H_2(SC(X))$ is non-trivial.

Since $p^{-1}(\II\times\{ 0\})$, $p^{-1}(\II\times\{ 1/2\})$,
$p^{-1}(\II\times\{ 1\})$ are strong deformation retracts of
$Y_0$, $Y_0\cap Y_1$ and $Y_1$ respectively, 
the homotopy types of $Y_0$, $Y_1$ and $Y_0\cap Y_1$
have the same homotopy type as  $p^{-1}(\II\times \{ 0\})$. 
We denote the deformation retractions by $r_0:Y_0\to p^{-1}(\II\times\{
0\})$ and $r_1:Y_1\to p^{-1}(\II\times\{ 1\})$.

Choose a point $x^{\#}\in X$ and form a one point union $(X,x^{\#})\vee
(\II ,0)$ under the identification of $x^{\#}$ and $0$. Let $X_n$ s be copies
of $(X,x^{\#})\vee (\II ,0)$ and $x_n$ s copies of $1\in \II$. 
Then the space $p^{-1}(\II\times \{ 0\})$ has the same homotopy type 
$Y = \widetilde{\bigvee}_{n\in\mathbb{N}}(X_n,x_n)$. 
Hence $(X_n,x_n)$ is locally simply connected and first countable at $x_n$. 
Lemma~\ref{lmm:first} implies that 
$\pi _1(Y) \simeq \pmc{$\times$}\, \, \; _{n\in\mathbb{N}}\pi _1(X_n,x_n)$.

Since $X$ is not simply connected, we can find
an essential loop $f$ in $X$ whose base point is $x^{\#}$. 
Observe that $p^{-1}(\{ P\})$ is a copy
of $X$ for each point $P$ on $A_1B_1A_2B_2\cdots$. 
A point $P$ on $A_1B_1A_2B_2\cdots$ is written as $P = (x,y)$ for
$x,y\in \II$. Define 
\[
 f_P(t) =\left\{ \begin{array}{ll}
(3xt,y), &\quad\mbox{ for } 0\le t\le 1/3 \\ 
(P,f(3(t-1/3))), & \quad\mbox{ for } 1/3 \le t \le 2/3 \\
(3(1-t)x,y), & \quad\mbox{ for } 2/3 \le t \le 1.
 \end{array}
\right. 
\]

Then for $n\ge 1$, $f_{A_n}$ is a loop in $p^{-1}(\II\times\{ 0
\})\subseteq Y_0$ with the base point $A$ and 
$f_{B_n}$ one in $p^{-1}(\II\times\{ 1\})\subseteq Y_1$ with the base
point $B$ and $f_{C_n}$ one in $p^{-1}(\II\times\{ 1/2 \})\subseteq Y_0\cap
Y_1$ with the base point $C$ respectively. 
Since the images of $f_{C_n}$ s converge to $C$, we have two loops $g_0 =
 f_{C_1}f_{C_2}^- f_{C_3}f_{C_4}^-\cdots $ and $g_1 =
 f_{C_1}^-f_{C_2} f_{C_3}^-f_{C_4}\cdots $ in $Y_0\cap Y_1$. 
(These infinite concatenations make sense, since the ranges of loops
converge to $C$.)

Observe that 
$r_{0*}\circ i_{0*}([f_{C_1}]) = [f_{A_1}]$,  $r_{1*}\circ
i_{1*}([f_{C_1}]) = [f_{B_1}]$, $r_{0*}\circ i_{0*}([f_{C_{2n}}]) =
[f_{A_{n+1}}] = r_{0*}\circ i_{0*}([f_{C_{2n+1}}])$ and
$r_{1*}\circ i_{1*}([f_{C_{2n-1}}])= [B_n] =r_{1*}\circ
i_{1*}([f_{C_{2n}}])$ for each natural number $n$.

Since we have homotopies from $f_{A_{n+1}}^-f_{A_{n+1}}$
to the constant $A$ and the images of the homotopies
converge to $A$, it follows that $r_{0*}\circ i_{0*}([g_1]) = [f_{A_1}]$ and
$r_{0*}\circ i_{0*}([g_2]) =[f_{A_1}^-]$. 
Hence $i_{0*}([g_0g_1]) = e$. Similarly,
$r_{1*}\circ i_{1*}([g_0]) = e$ and $r_{1*}\circ i_{1*}([g_1]) = e$ and
hence $r_{1*}\circ i_{1*}([g_0g_1]) = e$. 
Now we have $i_{0*}([g_0g_1]_s) = 0$ and
$i_{1*}([g_0g_1]_s) = 0$, i.e. $h([g_0g_1]_s) = 0$.

It suffices to show that $[g_0g_1]_s$ is non-zero, i.e. that
$[g_0g_1]$ does not belong to the commutator subgroup of $\pi _1(Y_0\cap
Y_1)$. The isomorphism from $\pi _1(Y_0\cap Y_1)$ to
$\widetilde{\bigvee}_{n\in\mathbb{N}}(X_n,x_n)$ maps 
$[g_0g_1]$ to $c_1c_2^{-1}c_3c_4^{-1}\cdots c_1^{-1}c_2c_3^{-1}c_4\cdots$,
where $c_n$ is the letter corresponding to $[f_{C_n}]$. 
To show the
conclusion by contradiction, suppose that $c_1c_2^{-1}c_3c_4^{-1}\cdots
c_1^{-1}c_2c_3^{-1}c_4\cdots$ belongs to the commutator subgroup. Then,
by \cite[Lemma 4.11]{E:free} there exist non-empty reduced words
$U_1,\cdots, U_{2m}$ such that $c_1c_2^{-1}c_3c_4^{-1}\cdots
c_1^{-1}c_2c_3^{-1}c_4\cdots = U_1\cdots U_{2m}$, where $U_1,\cdots, U_{2m}$
is of the canonical commutator form, i.e. there are $j_l,k_l$ 
such that $\{j_1,\cdots j_m\}\cup \{ k_1,\cdots ,k_m\} = \{ 1,\cdots
,2m\}$, $U_{j_l} = U_{k_l}^{-1}$ and the reduced word
$c_1c_2^{-1}c_3c_4^{-1}\cdots c_1^{-1}c_2c_3^{-1}c_4\cdots $ is obtained
by multiplying the rightmost elements $U_i$ and the 
leftmost elements of $U_{i+1}$ at most $(2m\! -\! 1)$-times. 
Therefore, $W_{2m}$ is of infinite length and is well-ordered from
the left to the right,
and hence there exists $U_i$ which is of 
infinite length and is well-ordered from the right to the left. But this
is impossible, because $c_1c_2^{-1}c_3c_4^{-1}\cdots
c_1^{-1}c_2c_3^{-1}c_4\cdots $ is well-ordered from the left to the
right.

Next we show the second statement (2). Suppose that $\pi _1(X)$ and $\pi
_2(X)$ are trivial. Consider another part of the Mayer-Vietoris
sequence: 
\[
H_2(Y_0)\oplus H_2(Y_1) \longrightarrow 
H_2(SC(X)) \overset{\partial}\longrightarrow 
H_1(Y_0\cap Y_1). 
\]
By $\pi _1(Y_0\cap Y_1) \simeq \pmc{$\times$}\, \, \; _{n\in\mathbb{N}}\pi _1(X_n,x_n)$, 
we conclude that $\pi _1(Y_0\cap Y_1)$ is trivial. Hence 
$H_1(Y_0\cap Y_1)$ is trivial. Since $\pi _1(Y_0)$ is 
trivial, it follows that $H_2(Y_0)$ is isomorphic to $\pi _2(Y_0)$. 
Now we have $H_2(Y_0) = \pi _2(Y_0)\simeq \Pi _{n\in\mathbb{N}} \pi _2(X_n,x_n) 
= \{ 0\}$ by \cite[Theorem 1.1]{EK:aloha}. 
Similarly, $H_2(Y_1) = 0$ and $H_2(Y_0)\oplus H_2(Y_1) = \{ 0\}$. 
Now the above exact sequence implies that $H_2(SC(X))$ is trivial.  
\qed

\smallskip
We denote the commutator $aba^{-1}b^{-1}$ by $[a,b]$.

\smallskip
{\it Alternative proof of\/} Corollary~\ref{cor:circle}.
For the case $X = S^1$ we take $c_n$ as the generator of the fundamental
group of $X_{C_n}$, which is isomorphic to $\mathbb{Z}$. 
As in the preceding proof of Theorem~\ref{thm:main1},
it suffices to show that the element $c = c_1c_2^{-1}c_3c_4^{-1}\cdots
c_1^{-1}c_2c_3^{-1}c_4\cdots$ does not belong to the commutator
subgroup of the group $\pi _1(Y_0\cap Y_1)$. To 
prove this by contradiction, suppose that
$c$ belongs to the commutator subgroup, i.e. 
$c$ is a product of $m$ commutators for some $m$.

Consider natural homomorphism $f: \pi_1(Y_0 \cap Y_1) \to
\pi_1(\bigvee_{1\leq i \leq 2m+2}(X_{C_i}, C_i)),$ where $X_{C_i} =
S^1$. The group $\pi_1(\bigvee_{1\leq i \leq
2m+2}(X_{C_i}, C_i))$ is 
a free group with $2m+2$-generators
$\langle c_1,c_2,\cdots ,c_{2m+1},c_{2m+2}\rangle.$ We have

$$f(c) =
c_1c_2^{-1}\cdots c_{2m+1}c_{2m+2}^{-1}c_1^{-1}c_2\cdots
c_{2m+1}^{-1}c_{2m+2}.$$

Let $d_1 = c_1, d_2 = c_2^{-1},$ $d_{2k-1} =
c_{2k-2}^{-1}c_{2k-3}\cdots c_2^{-1}c_1c_{2k-1}$ and
$d_{2k} = c_{2k}^{-1}c_{2k-1}$.

It is easy to prove by induction the
equality $c_1c_2^{-1}\cdots c_{2k-1}c_{2k}^{-1} c_1^{-1}c_2\cdots
c_{2k-1}^{-1}c_{2k} = [d_1,d_2]\cdots [d_{2k-1},d_{2k}]$.

Since $( d_1,d_2,\cdots ,d_{2m+1},d_{2m+2} )$ is obtained by a
Nielsen transformation \cite[p.5]{LS} from $( c_1,c_2,\cdots
,c_{2m+1},c_{2m+2})$, the set $\{ d_0,d_1,\cdots ,d_{2m},d_{2m+2}\}$
generates the free group $\langle  c_1,c_2,\cdots
,c_{2m+1},c_{2m+2}\rangle$.
It follows from this and by \cite[Proposition 6.8, p.55]{LS} (see also
\cite{Cu}, p.137)  that $f(c)$ 
cannot be presented as a product of less than 
$m+1$ commutators. This contradicts our assumption. \qed

\section{A PL Model for $SC(S^1)$ and Some Related Constructions}
In this section we demonstrate piecewise linear constructions which are
similar to $SC(S^1)$, using parameters for oscillations of a
tube. Actually we prove in Theorem~\ref{thm:main2} that they are
homotopy equivalent to the point, $SC(S^1)$, or $\mathcal{G}$ depending
on their parameters.

For $0\le y \le 1$ and $\vep\ge 0$ with $0 < y+\vep\le 1$, we construct a
space $S(y,\vep )\subseteq \RR^3$ as follows. 
Consider the following points on $\II ^2$ for $n\in \NN$ (see
Figure 2), where we regard $\II ^2 \subseteq \RR ^2$ as $\II ^2\times \{
0\}$:

$A_n = (\displaystyle{\frac{1}{2n-1}}, \ 0)$, \;
$B_n = (\displaystyle{\frac{1}{2n}}, \ 1)$, \; 
$C_n = (\displaystyle{\frac{1}{2n-1}, \ y + \frac{\vep}{2n-1}})$,

$D_n = (\displaystyle{\frac{1}{2n},\ 1 -y - \frac{\vep}{2n}})$, \;
$E_n = (\displaystyle{\frac{1}{2n-1},\ \frac{1}{2}(y +
\frac{\vep}{2n-1}}))$,

$F_n = (\displaystyle{\frac{1}{2n},\ \frac{1}{2} (2 - y -
\frac{\vep}{2n}}))$.

Let $\ov{E}_n$ and $\ov{F}_n$ be points on the plane
$\{(z, x, y)\in \RR^3 |\ z = \frac{1}{2}x\}$
the projections of which to the plane
${\RR}^2$ are points $E_n$ and $F_n$ respectively, i.e.,

$\ov{E}_n = (\displaystyle{\frac{1}{2n-1},\ \frac{1}{2}(y +
\frac{\vep}{2n-1}), \ \frac{1}{2(2n-1)}})$, \;
$\ov{F}_n = (\displaystyle{\frac{1}{2n},\ \frac{1}{2} (2 - y -
\frac{\vep}{2n}), \ \frac{1}{4n}})$.

Let $H_{2n-1}$ be the convex hull of the points $A_n, B_n, C_n, D_n,
\ov{E}_n$ and $\ov{F}_n$ and $H_{2n}$ the convex hull of the
points $A_{n+1}, B_n, C_{n+1}, D_n, \ov{F}_n$ and $\ov{E}_{n+1}$.

Let $H_{\infty}$ be the set $\cup_{n=1}^{\infty}H_n$ and $\partial
H_{\infty}$ its boundary. Let $\Delta A_1C_1\ov{E}_1$ be an open triangle in 
 $\partial H_{\infty}$. Finally,
 define $S(y, \vep )$ to be the subspace 
$(\II ^2\times \{ 0\} )\cup \partial H_{\infty}\setminus \Delta
A_1C_1\ov{E}_1$ of $\RR^3.$

\medskip
\setlength{\unitlength}{0.7mm}
\begin{picture}(130,130)
\put(120,3.5){\makebox(1,1)[r]{$\bullet$}}\put(120,0){\makebox(1,1)[l]{$A_1$}}
\put(41.75,3.5){\makebox(1,1)[r]{$\bullet$}}\put(42,0){\makebox(1,1)[l]{$A_2$}}
\put(21,3.5){\makebox(1,1)[r]{$\bullet$}}\put(21,0){\makebox(1,1)[l]{$A_3$}}
\put(0,3.5){\makebox(1,1)[r]{$\bullet$}}\put(0,0){\makebox(1,1)[l]{$A$}}

\put(60,123.5){\makebox(1,1)[r]{$\bullet$}}\put(60,126){\makebox(1,1)[l]{$B_1$}}
\put(30.5,123.5){\makebox(1,1)[r]{$\bullet$}}\put(30,126){\makebox(1,1)[l]{$B_2$}}
\put(0,123.5){\makebox(1,1)[r]{$\bullet$}}\put(0,126){\makebox(1,1)[l]{$B$}}

\thicklines
\put(120,4){\line(0,1){120}}
\put(0,4){\line(1,0){120}}
\put(0,124){\line(1,0){120}}
\put(0,4){\line(0,1){120}}

\put(120,94){\makebox(1,1)[r]{$\bullet$}}\put(122,94){\makebox(1,1)[l]{$C_1$}}
\put(60,49){\makebox(1,1)[r]{$\bullet$}}\put(60,47){\makebox(1,1)[l]{$D_1$}}
\put(40.5,74){\makebox(1,1)[r]{$\bullet$}}\put(42,76){\makebox(1,1)[l]{$C_2$}}
\put(30.5,57){\makebox(1,1)[r]{$\bullet$}}\put(30,55){\makebox(1,1)[l]{$D_2$}}
\put(21,70){\makebox(1,1)[r]{$\bullet$}}\put(21,72){\makebox(1,1)[l]{$C_3$}}

\put(120,50){\makebox(1,1)[r]{$\bullet$}}\put(122,47){\makebox(1,1)[l]{$E_1$}}
\put(60,86.5){\makebox(1,1)[r]{$\bullet$}}\put(60,88){\makebox(1,1)[l]{$F_1$}}
\put(40.5,39){\makebox(1,1)[r]{$\bullet$}}\put(40,37){\makebox(1,1)[l]{$E_2$}}
\put(30.5,90){\makebox(1,1)[r]{$\bullet$}}\put(30,93){\makebox(1,1)[l]{$F_2$}}
\put(21,37){\makebox(1,1)[r]{$\bullet$}}\put(21,35){\makebox(1,1)[l]{$E_3$}}

\thinlines
\put(60,124){\line(2,-1){60}}
\put(60,124){\line(-2,-5){20}}
\put(30,124){\line(1,-5){10}}
\put(30,124){\line(-1,-6){9}}

\put(59,87){\line(5,-3){60}}
\put(59,87){\line(-2,-5){19}}
\put(30,90){\line(1,-5){10}}
\put(30,90){\line(-1,-6){8.75}}

\put(60,49){\line(4,-3){60}}
\put(60,49){\line(-2,-5){18}}
\put(30,57){\line(1,-5){10.5}}
\put(30,57){\line(-1,-6){8.75}}


\end{picture}

\begin{center}
Figure 2: $S(1/2,1/4)$
\end{center}

\medskip
The first lemma is easy to prove and we
therefore
omit its proof.

\begin{lmm}\label{SC(S^1)}
Let 
$\vep , \vep ' \in (0,1)$. 
Then
the spaces $S(0, \vep )$ and $S(0, \vep ')$ are homeomorphic and 
$S(0, 1)$ is homotopy equivalent to $S(0, \vep )$. 
\end{lmm}

\begin{lmm}\label{lmm:JG}
If $0< y \le 1/2$ and $0 < y+\vep \le 1$, the space $S(y,\vep )$ is
 homotopy equivalent to $S(1/2, 0).$ 
\end{lmm}
\begin{proof}
It is easy to see that $S(y,\vep )$ and $S(y,0)$ are homeomorphic and so
 we only 
 need to
 prove that $S(y,0)$ for $0< y < 1/2$ and $S(1/2,0)$ are
 homotopy equivalent. 
 (Without any loss of generality we may assume that $y=1/3$.)
 
Since there might be some confusion regarding the homotopy equivalence,
 we explain this first. 
 Let $A_n$, $B_n$, $C_n$, $D_n$,... be the notation for $S(1/2,0)$ and
 $C'_n$, $D'_n$,...
 be the corresponding notation for $S(1/3,0)$. 
 
If we remove $\{ 0\}\times \mathbb{I}$ from $S(1/2,0)$ and $S(1/3,0)$,
 then the resulting spaces
 are homeomorphic, that is, $S(1/2,0)\sm 
\{ 0\}\times \mathbb{I}$ and $S(1/3,0)\setminus \{ 0\}\times \mathbb{I}$ are
 homeomorphic. However, this homeomorphism cannot be extended
 over to $S(1/2,0)$,
 since the homeomorphism maps $C_n$ to $C'_n$ and $D_n$ to
 $D'_n$, that is, upwards for $C_n$ and downwards for $D_n$,
 with respect to the $y$-coordinate. Conversely, if we construct a
 homotopy on $S(1/2,0)\sm \{ 0\}\times \mathbb{I}$ or 
$S(1/3,0)\setminus \{ 0\}\times \mathbb{I}$,
whose projection to the $y$-coordinate only depends on the
 $y$-coordinate on the domain, it extends on $SC(1/2,0)$ or
 $SC(1/3,0)$.

We define $\vp :S(1/2,0)\to S(1/3,0)$ and $\psi :S(1/2,0)\to S(1/3,0)$
 piecewise linearly as follows:

Let $\vp (x,y,0) = (x,y,0)$ and $\vp (x,y,z) = (x,y,\vp _2(x,y,z))$, for
$z>0$, where $\vp _2(x,y,z) > 0$ if and only if $z>0$ and there exists
 $z'>0$ such that $(x,y,z')\in S(1/3,0)$. 
Let 
\[
 \psi _1(y) = 
\left\{
\begin{array}{ll}
 3y/2, & \quad \mbox{ for }0\le y\le 1/3\\
 1/2, & \quad \mbox{ for }1/3\le y\le 2/3\\
 3y/2-1/2, & \quad \mbox{ for }2/3\le y\le 1. 
\end{array}
\right .
\]
and $\psi (x,y,z) = (\psi _0(x,y,z),\psi _1(y),\psi _2(x,y,z))$,
where 
$\psi _2(x,y,0) = 0$ and $\psi _2(x,y,z)>0$, for $z>0$ and $\psi _0(x,y,z)$ 
is defined as we explain using Figure 3 in the 
sequel.

Figure 3 demonstrates how $[\displaystyle{\frac{1}{2n+1}, \ \frac{1}{2n}}]
\times \II$ of $S(1/2,0)$ and $S(1/3,0)$ are mapped by $\vp$ and
 $\psi$.

\setlength{\unitlength}{0.7mm}
\begin{picture}(120,100)
\put(48,0){\makebox(1,1)[r]{$\bullet$}}
\put(0,0){\makebox(1,1)[r]{$\bullet$}}
\put(48,0){\makebox(1,1)[r]{$\bullet$}}
\put(0,96){\makebox(1,1)[r]{$\bullet$}}
\put(48,96){\makebox(1,1)[r]{$\bullet$}}
\put(0,48){\makebox(1,1)[r]{$\bullet$}}
\put(48,48){\makebox(1,1)[r]{$\bullet$}}
\put(-8,-2){\shortstack{$A_{n+1}$}}
\put(-8,46){\shortstack{$C_{n+1}$}}
\put(49,96){\shortstack{$B_{n}$}}
\put(49,48){\shortstack{$D_{n}$}}

\thicklines
\put(0,0.5){\line(1,0){48}}
\put(0,0.5){\line(1,1){48}}
\put(0,0){\line(0,1){96}}
\put(0,96.5){\line(1,0){48}}
\put(48,0){\line(0,1){96}}
\put(0,48.5){\line(1,1){48}}

\thinlines
\put(0,0.5){\line(3,4){48}}
\put(0,64.5){\line(1,0){24}}
\put(0,32.5){\line(3,4){48}}
\put(24,32.5){\line(1,0){24}}
\put(3,63){\line(0,-1){24}}
\put(6,63){\line(0,-1){20}}
\put(9,63){\line(0,-1){16}}
\put(12,63){\line(0,-1){12}}
\put(15,63){\line(0,-1){8}}
\put(18,63){\line(0,-1){4}}
\put(21,63){\line(0,-1){1}}

\put(27,34){\line(0,1){1}}
\put(30,34){\line(0,1){4}}
\put(33,34){\line(0,1){8}}
\put(36,34){\line(0,1){12}}
\put(39,34){\line(0,1){16}}
\put(42,34){\line(0,1){20}}
\put(45,34){\line(0,1){24}}


\put(68,48){\makebox(1,1)[r]{$\bullet$}}
\put(116,48){\makebox(1,1)[r]{$\bullet$}}

\put(116,0){\makebox(1,1)[r]{$\bullet$}}
\put(68,0){\makebox(1,1)[r]{$\bullet$}}
\put(116,0){\makebox(1,1)[r]{$\bullet$}}
\put(68,96){\makebox(1,1)[r]{$\bullet$}}
\put(116,96){\makebox(1,1)[r]{$\bullet$}}
\put(60,-2){\shortstack{$A'_{n+1}$}}
\put(60,32){\shortstack{$C'_{n+1}$}}
\put(118,64){\shortstack{$D'_{n}$}}
\put(118,96){\shortstack{$B'_{n}$}}
\put(88,67){\shortstack{$N'_{n}$}}
\put(93,28){\shortstack{$M'_{n}$}}


\thicklines
\put(68,0.5){\line(1,0){48}}
\put(68,0.5){\line(3,4){48}}
\put(68,0){\line(0,1){96}}
\put(68,96.5){\line(1,0){48}}
\put(116,0){\line(0,1){96}}
\put(68,32.5){\line(3,4){48}}

\thinlines
\put(84,32.5){\line(3,4){24}}
\put(68,0.5){\line(1,2){16}}
\put(108,64.5){\line(1,2){8}}
\put(116,96){\line(-1,-2){16}}
\put(100,64){\line(-3,-4){24}}
\put(76,32){\line(-1,-2){8}}

\put(68,64.5){\line(1,0){24}}
\put(92,32.5){\line(1,0){24}}
\put(71,63){\line(0,-1){24}}
\put(74,63){\line(0,-1){20}}
\put(77,63){\line(0,-1){16}}
\put(80,63){\line(0,-1){12}}
\put(83,63){\line(0,-1){8}}
\put(86,63){\line(0,-1){4}}
\put(89,63){\line(0,-1){1}}

\put(95,34){\line(0,1){1}}
\put(98,34){\line(0,1){4}}
\put(101,34){\line(0,1){8}}
\put(104,34){\line(0,1){12}}
\put(107,34){\line(0,1){16}}
\put(110,34){\line(0,1){20}}
\put(113,34){\line(0,1){24}}
\end{picture}

\smallskip
\begin{center}
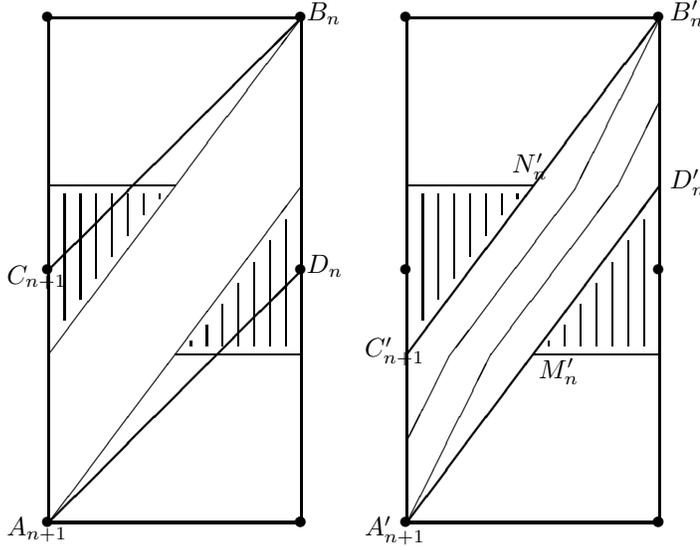
Figure 3 : Parts of $S(1/2,0)$ and $S(1/3,0)$
\end{center}
\medskip

First we explain the map
$\psi$. The two shadowed triangles are mapped to
 $C_{n+1}$ or $D_n$, respectively. Accordingly,
 the segments
 $B'_{n}C'_{n+1}$ and $D'_nA'_{n+1}$ are mapped onto $B_nC_{n+1}$ and
 $D_nA_{n+1}$ respectively. 
The segments $N'_nD_n$ and $C'_{n+1}M'_n$ are mapped bijectively to
 $C_{n+1}D_n$.

Next we explain the map
$\vp\psi$. The two shadowed triangles are mapped to 
$\vp (C_{n+1})$ or $\vp (D_n)$, which are 
the dotted point. The two 
bending segments are mapped onto $C'_{n+1}B'_n$ or $A'_{n+1}D'_n$.

Last we explain the map
$\psi\vp$. The two shadowed triangles are mapped to 
$C_{n+1}$ or $D_n$. The two segments having slope greater than $1$ are
 mapped to $C_{n+1}B_n$ or $A_{n+1}D_n$.

We have a homotopy $H(x,y,z,t)$ on $S(1/2,0)\sm (\{ 0\}\times \II )$
such that:
\begin{itemize}
\item[(1)] $H(x,y,z,0) = (x,y,z)$ and $H(x,y,z,1)=\psi\vp(x,y,z)$;  
\item[(2)] for the $y$-coordinate $H_1(x,y,z,t)$ of $H(x,y,z,t)$, 
\[
 H_1(x,y,z,t) = 
\left\{
\begin{array}{ll}
 y+yt/2, & \quad \mbox{ for }0\le y\le 1/3\\
 y+t/2-yt, & \quad \mbox{ for }1/3\le y\le 2/3\\
 y-t/2+yt/2, & \quad \mbox{ for }2/3\le y\le 1\mbox{;} 
\end{array}
\right .
\]
\item[(3)] $H(*,*,*,t)$ maps $p^{-1}([\displaystyle{\frac{1}{n+1}, 
\ \frac{1}{n}}] \times \II )$ onto itself for each $n$. 
\end{itemize}
Then we can extend $H(*,*,*,t)$ to $S(1/2,0)$ uniquely and
 continuously.

Concerning $S(1/3,0)$ with $\vp\psi$,
we have a homotopy with the same
 properties as above and we now see
 that $S(1/2,0)$ and $S(1/3,0)$ are
 homotopy equivalent.  
\end{proof}
\begin{thm}\label{thm:main2}
Suppose that $0\le y \le 1$, $\vep \ge 0$ and $0 < y+\vep\le 1$.  
Then the following assertions hold:
\begin{itemize}
\item[(1)] For every $1/2 <y\le 1$, the spaces $S(1, 0)$ and $S(y,\vep )$  are
           contractible;
\item[(2)] For $y=0$, the space $S(y, \vep )$  is homotopy equivalent to
           $SC(S^1)$; and
\item[(3)] For every $0< y \le 1/2$, the space $S(y, \vep )$ is homotopy equivalent to the space $\mathcal{G}$ . 
\end{itemize}
\end{thm}

\begin{proof}
The statements (1) and (2) are easy to verify. 
Therefore 
we shall only prove (3).

By Lemma~\ref{lmm:JG},
it suffices to show that $S(1/2,1/4)$ is homotopy
 equivalent to the space $\mathcal{G}$. 
Let $\Delta$ be 
the
half-open triangle, defined as 
$\Delta = \{ (x,y)\ |\ x\in (0, 1],y\in (-x/4 +1/2, x/4 + 1/2)\}$. 
Then $p^{-1}(\II ^2\setminus \Delta)$ is a strong deformation retract of
 $S(1/2,1/4)$.

Identifying 
$\{(x, y)\ |\ y= a + (1-a)x/4, x\in \II\}$ as one point for $a\in [1/2,1]$ 
and
$\{(x, y)\ |\ y= a - ax/2, x\in \II \}$ as one point for $a\in [0,1/2]$,
 we get the quotient space of $p^{-1}(\II ^2\setminus \Delta)$, 
which is homeomorphic to $\mathcal{G}$. 
Now the homotopy equivalence between $p^{-1}(\II ^2\setminus \Delta)$
 and $\mathcal{G}$ is evident and so $S(1/2,1/4)$ is indeed
homotopy equivalent to $\mathcal{G}$.
\end{proof}
\begin{rmk} 
{\rm
The space $SC(S^1)$ is simply connected 
(see \cite{EKR2}),
 whereas the space $\mathcal{G}$ is not simply connected 
 (see \cite{GJ}). 
We remark that
$H_2(\mathcal{G}) = \{ 0\}$, which contrasts with
 Theorem~\ref{cor:circle}.

To show this, we introduce some notation. Since the cone $C(X)$ over
the space $X$ is the quotient space of $X\times \II$,
 obtained by identifying $X\times \{ 1\}$ to a point, 
we let $p:X \times \II \to C(X)$ be the canonical projection.

For a subset $A$ of $\II$, let $C_A(X) = p(X \times A) \subset C(X)$.
Let $\mathbb{H}_1$ and $\mathbb{H}_2$ be copies of the Hawaiian earring
 $\mathbb{H}$ and $\mathcal{G} = C(\mathbb{H}_1)\vee C(\mathbb{H}_2)$ be the one
 point union of $C(\mathbb{H}_1)$ and $C(\mathbb{H}_2)$ defined in
 Section~\ref{sec:preliminary}. 
Let $X_1$ be the disjoint union of $C_{(1/3,1]}(\mathbb{H}_1)$ and
 $C_{(1/3,1]}(\mathbb{H}_2)$ and $X_2$ be $C_{[0,2/3)}(\mathbb{H}_1)\vee 
C_{[0,2/3)}(\mathbb{H}_2)$.

Then  $\mathcal{G} = X_1\cup X_2$ and
we have
the following part of the
 Mayer-Vietoris sequence:
\[
H_2(X_1) \oplus H_2(X_2) \longrightarrow 
H_2(\mathcal{G}) \overset{\partial}\longrightarrow 
H_1(X_1\cap X_2)\overset{h}\longrightarrow 
H_1(X_1)\oplus H_1(X_2). 
\]
Obviously,
$H_2(X_1) = \{ 0\}$. 
Since $X_2$ is homotopy equivalent to
 $\mathbb{H}_1\vee \mathbb{H}_2$ which is a 1-dimensional compact metric
 space, $H_2(X_2)$ is trivial \cite{CF}. 
 Hence $\partial$ is injective. 
We observe that $X_1\cap X_2$ is the disjoint union of
 $C_{(1/3,2/3)}(\mathbb{H}_1)$ and $C_{(1/3,2/3)}(\mathbb{H}_2)$.

Since $C_{[1/3,2/3)}(\mathbb{H}_1)$ and $C_{[1/3,2/3)}(\mathbb{H}_2)$ are
 retracts of $C_{[0,2/3)}(\mathbb{H}_1)\vee C_{[0,2/3)}(\mathbb{H}_2)$
 and are homotopy equivalent to $C_{(1/3,2/3)}(\mathbb{H}_1)$ and
 $C_{(1/3,2/3)}(\mathbb{H}_2)$ respectively, it follows that $h$ is
 injective. Therefore we obtain that $H_2(\mathcal{G}) = \{ 0\}$. 
}
\end{rmk}
\begin{prb}
Does there exist a finite-dimensional non-contractible Peano
continuum all homotopy groups of which are trivial?
\end{prb}

\begin{rmk} 
Recently we have strengthened Theorem~\ref{thm:main1}(2) 
by proving the following: 
If $X$ is 
simply connected, 
then $\pi _2(SC(X))$ is trivial. 
We have proved earlier  that
$SC(X)$ is also simply connected \cite{EKR2}. 
Therefore by
Theorem~\ref{thm:main1}(1)
the following statements are equivalent 
for any path-connected space $X$: 
\begin{itemize}
\item[(1)] $X$ is simply connected; 
\item[(2)] $\pi _2(SC(X))$ is trivial; and
\item[(3)] $H_2(SC(X))$ is trivial. 
\end{itemize}
\end{rmk}

{\large Acknowledgements}

This paper was presented (by the third author) 
at
the special session 
{\sl Topology of Continua}
of the AMS Spring Central Section Meeting 
in Lubbock, Texas (April 8-10, 2005).
In June 2005,  during his visit to Ljubljana,
J. Dydak informed the third author that since
then together with A. Mitra they have
obtained an independent proof of
Corollary 3.2 
(their manuscript is not yet available). 
We were supported in part by the Japanese-Slovenian research grant
BI--JP/03--04/2 and the Slovenian Research Agency research project
No. J1--6128--0101--04 and program P1-0292-0101-04.

We thank the referee for several useful comments and suggestions.

\end{document}